\documentclass{elsarticle}

	\usepackage{amsfonts,amsmath,amsthm}
	\usepackage{amssymb}
	\usepackage[numbers]{natbib}
	\usepackage{color}
	\usepackage{graphicx}
	\usepackage[hmargin=3cm,vmargin={3.5cm,4cm}]{geometry}
	\usepackage{booktabs}

	\newtheorem{te}{Theorem}[section]
	\newtheorem{ex}{Example}[section]
	\newtheorem{definition}{Definition}[section]
	\newtheorem{os}{Remark}[section]
	\newtheorem{prop}{Proposition}[section]
	\newtheorem{lem}{Lemma}[section]

\begin{document}

		\title{Hilfer--Prabhakar Derivatives and Some Applications}

		\author{Roberto Garra$^1$}
			\address{${}^1$Dipartimento di Scienze di Base e Applicate per l'Ingegneria,
			Sapienza Universit\`a di Roma.}
		\author{Rudolf Gorenflo$^2$}
			\address{${}^2$Department of Mathematics and Informatics, Free University of
			Berlin.}
		\author{Federico Polito$^3$}
			\address{${}^3$Dipartimento di Matematica ``G.\ Peano'', Universit\`a degli Studi di
			Torino.}
		\author{\v{Z}ivorad Tomovski$^4$}
			\address{${}^4$Department of Mathematics, Sts.\ Cyril and Methodius
			University, Skopje.}

		\date{\today}

		\begin{abstract}

			We present a generalization of Hilfer derivatives in which Riemann--Liouville
			integrals are replaced by more general Prabhakar integrals. We analyze and discuss its properties.
			Furthermore, we show some applications of these generalized Hilfer--Prabhakar
			derivatives in classical equations of mathematical physics such as the heat and
			the free electron laser equations, and in difference-differential
			equations governing the dynamics of generalized renewal
			stochastic processes.

			\bigskip			
			
			\textit{Keywords}: Hilfer--Prabhakar derivatives, Prabhakar Integrals,
			Mittag--Leffler functions, Generalized Poisson Processes
		
		\end{abstract}

		\maketitle

		\section{Introduction}
		
			In the recent years fractional calculus has gained much interest mainly thanks to the increasing presence of
			research works in the applied sciences considering models based on fractional operators.
			Beside that, the mathematical study of fractional calculus has proceeded, leading to intersections
			with other mathematical fields such as probability and the study of stochastic processes.
			
			In the literature, several different definitions of fractional
			integrals and derivatives are present. Some of them such as the Riemann--Liouville integral, the Caputo and the Riemann--Liouville
			derivatives are thoroughly studied and actually used in applied models. Other less-known definitions such as
			the Hadamard and Marchaud derivatives are mainly subject of mathematical investigation
			(the reader interested in fractional calculus in general can consult one of the classical
			reference texts such as \cite{samko,kila,podlubny}).
			
			In this paper we introduce a novel generalization of derivatives of 
			both Riemann--Liouville and Caputo types and show the effect of using it in equations
			of mathematical physics or related to probability. In order to do so,
			we start from the definition of generalized fractional derivatives given by R.~Hilfer \cite{hill1}.
			The so-called Hilfer fractional derivative is in fact a very convenient way to generalize
			both definitions of derivatives as it actually interpolates them
			by introducing only one additional real parameter $\nu\in [0,1]$. The further generalization that we are going to discuss
			in this paper is given by replacing Riemann--Liouville fractional integrals with Prabhakar integrals 
			in the definition of Hilfer derivatives. We recall that the Prabhakar integral \cite{Prab} is obtained by modifying 
			the Riemann--Liouville integral operator by extending its kernel with a three-parameter Mittag--Leffler function,
			a function which extends the well-known two-parameter Mittag--Leffler function. This latter function
			was used by J.D.\ Tamarkin \cite{tam} in 1930
			and later gained importance in treating problems of fractional
			relaxation and oscillation, see e.g.\ \cite{marev} for a grand
			survey. 
			The Hilfer--Prabhakar derivative (which contains the Hilfer derivative as a specific case)
			interpolates the Prabhakar derivative, first introduced in \cite{Kil} and its Caputo-like 
			regularized counterpart. In Section \ref{barabasi}, we study some of its properties and,
			in Section \ref{appll}, we discuss some related applications of interest in mathematical physics and probability.
			We commence by analyzing the time-fractional heat equation involving Hilfer--Prabhakar derivatives. 
			We discuss the main differences between the solution of the Cauchy problems
			involving the non-regularized and the regularized operators.
			Another integro-differential equation of interest for applications is the free electron laser (FEL)
			integral equation \cite{dat}. This equation arises in the description of the unsatured behavior of the free electron laser.
			Several generalizations of this equation involving other fractional operators have been studied in literature 
			(see e.g. \cite{kil1}).
			Taking inspiration from these works, we study a FEL-type integro-differential equation involving Hilfer--Prabhakar fractional
			derivatives.			
			A further application that we study in Section \ref{poipoi} regards the derivation of a renewal point
			process which is in fact a direct generalization of the classical homogeneous Poisson process
			and the time-fractional Poisson process. The connection with Hilfer--Prabhakar derivatives
			comes from the fact that the state probabilities are
			governed by time-fractional difference-differential equations involving Hilfer--Prabhakar derivatives.
			We give a complete discussion of the main properties of this process, providing also the explicit form of the probability 
			generating function, a subordination representation in terms of a time-changed Poisson process, and its renewal structure.

			Before introducing the definition of the Hilfer--Prabhakar derivatives, in Section \ref{prel}, in order
			to make the paper self-contained, we recall some basic definitions and results of fractional calculus.
			Section \ref{fer} presents instead the definition of Hilfer derivatives along with some of their properties.

		\section{Preliminaries on fractional calculus}

			\label{prel}
			Before introducing the non-regularized and regularized Hilfer--Prabhakar differential operators,
			for the reader's convenience, in this section we recall some definitions of classical fractional operators.
			In particular, the classical Riemann--Liouville derivative and its regularized operator (the so-called
			Caputo derivative) will be described.			
			However, in order to gain more insight on fractional calculus the reader can consult the classical
			reference books \cite{samko,podlubny,kila,diethelm}.
	
			\begin{definition}[Riemann--Liouville integral]
			    Let $f \in L^1_{\text{loc}}[a,b]$, where $-\infty \le a < t < b \le \infty$, be a locally integrable real-valued function.
			    The Riemann--Liouville integral is defined as
			    \begin{align}
			        \label{rlint}
			        I^{\alpha}_{a^+} f(t) & =\frac{1}{\Gamma(\alpha)}\int_{a}^t\frac{f(u)}{%
			        (t-u)^{1-\alpha}}du
			        = (f \ast K_\alpha )(t), \qquad \alpha > 0,
			    \end{align}
			    where $K_\alpha (t) = t^{\alpha-1}/\Gamma(\alpha)$.
			\end{definition}
	
			\begin{definition}[Riemann--Liouville derivative]
			    Let $f \in L^1[a,b]$, $-\infty \le a < t < b \le \infty$, and $f \ast K_{m-\alpha} \in W^{m,1}[a,b]$, $m = \lceil \alpha \rceil$,
			    $\alpha>0$,
			    where $W^{m,1}[a,b]$ is the Sobolev space defined as
			    \begin{align}
			        W^{m,1}[a,b] = \left\{ f \in L^1[a,b] \colon \frac{d^m}{dt^m} f \in L^1[a,b] \right\}.
			    \end{align}
			    The Riemann--Liouville derivative of order $\alpha >0$ is defined as
			    \begin{align}
			        \label{rlder}
			        D^\alpha_{a^+}f (t) = \frac{d^m}{dt^m}I_{a^+}^{m-\alpha}f(t) = \frac{1}{%
			        \Gamma(m-\alpha)} \frac{d^m}{dt^m} \int_{a}^t (t-s)^{m-1-\alpha}f(s) ds.
			    \end{align}
			\end{definition}

			For $ n\in \mathbb{N}$, we denote by $AC^{n}\left[a,b\right]$ the
			space of real-valued functions $f\left( t\right) $ which have
			continuous derivatives up to order $n-1$ on $\left[ a,b\right] $
			such that $f^{\left(n-1\right) }\left(t\right)$ belongs to the space of absolutely continuous functions
			$AC\left[a,b\right]:$
			\begin{equation}
				AC^{n}\left[a,b\right] =\left\{ f:\left[a,b\right] \rightarrow
				\mathbb{R}\colon\frac{d^{n-1}}{dx^{n-1}}f \left( x\right) \in AC\left[
				a,b\right] \right\} .
			\end{equation}

			\begin{definition}[Caputo derivative]
			    Let $\alpha>0$, $m = \lceil \alpha \rceil$, and $f \in AC^m[a,b]$.
			    The Caputo derivative of order $\alpha>0$ is defined as
			    \begin{equation}
			        \label{Capu}
			        {}^C D^{\alpha}_{a^+}f(t)= I_{a^+}^{m-\alpha}\frac{d^m}{dt^m}f(t)= \frac{1%
			        }{\Gamma(m-\alpha)}\int_a^{t}(t-s)^{m-1-\alpha}\frac{d^m}{ds^m}f(s) \, ds.
			    \end{equation}
			\end{definition}

			In the space of the functions belonging to $AC^m[a,b]$ the following
			relation between Riemann--Liouville and Caputo derivatives holds
			\cite{hill}.

			\begin{te}
			    \label{gianduia}
			    For $f \in AC^m[a,b]$, $m=\lceil \alpha \rceil$, $\alpha\in\mathbb{R}^+\backslash\mathbb{N}$,
			    the Riemann--Liouville derivative of order $\alpha$ of $f$
			    exists almost everywhere and it can be written as
			    \begin{align}
			        \label{perepe}
			        D_{a^+}^\alpha f(t) = {}^C D_{a^+}^\alpha f(t) + \sum_{k=0}^{m-1} \frac{(x-a)^{k-\alpha}}{\Gamma(k-\alpha+1)} f^{(k)}(a^+).
			    \end{align}
			\end{te}

			The above theorem gives the set of functions where the
			Riemann--Liouville derivative can be regularized. Moreover, if
			$f(t)\in AC^m[a,b]$, we have (see e.g. \cite{AD})
			\begin{align}
			    \label{purupu}
			    \lim_{t \to a^+} \frac{d^k}{dt^k} I_{a^+}^{m-\alpha} f(t) = 0, \qquad \forall \quad 0 \le k \le m-1.
			\end{align}
			Indeed, taking the Laplace transform of both sides of \eqref{perepe} the equality holds if \eqref{purupu} is true.

			See also \citet{gorgor} for more information. In that paper (page 228) the \emph{Caputo fractional derivative} was baptized
		    so, ad late it gained rising popularity by Podlubny's book \citep{podlubny} on fractional differential equations.

		\section{Hilfer derivatives}

			\label{fer}
			In a series of works (see \cite{hill} and the references therein), R.\ Hilfer
			studied applications of a generalized fractional operator having the Riemann--Liouville
			and the Caputo derivatives as specific cases (see also \cite{tom,tom2}).
			\begin{definition}[Hilfer derivative]
			    Let $\mu\in (0,1)$, $\nu \in[0,1]$, $f \in L^1[a,b]$,
			    $-\infty \le a < t < b \le \infty$, $f \ast K_{(1-\nu)(1-\mu)} \in AC^1[a,b]$. The Hilfer derivative is defined as
			    \begin{equation}
			        \label{hil}
			        D^{\mu,\nu}_{a^+}f(t)=\left(I_{a^+}^{\nu(1-\mu)} \frac{d}{dt}%
			        (I_{a^+}^{(1-\nu)(1-\mu)}f)\right)(t),
			    \end{equation}
			\end{definition}
			Hereafter and without loss of generality we set $a=0$.
			The generalization \eqref{hil}, for $\nu = 0$,
			coincides with the Riemann--Liouville derivative \eqref{rlder} and for $\nu=
			1$ with the Caputo  derivative \eqref{Capu}.  A relevant point in the
			following discussion regards the initial conditions that should be
			considered  in order to solve  fractional Cauchy problems involving Hilfer
			derivatives.  Indeed, in view of the Laplace transform of the Hilfer
			derivative (\cite{tom}, formula (1.6))
			\begin{align}
				\mathcal{L}[D^{\mu,\nu}_{0^+}f](s)=s^{\mu}\mathcal{L}[f](s)-s^{\nu(\mu-1)}(I^{(1-\nu)(1-\mu)}_{0^+}f)(0^+),
			\end{align}
			it is clear that the initial conditions that must be considered are of the
			form  $(I_{0^+}^{(1-\nu)(1-\mu)}f)(0^+)$,  i.e.\ on the initial value of the
			fractional integral of order $(1-\nu)(1-\mu)$.  These initial conditions do
			not have a  clear physical meaning unless $\nu=1$.  In order to obtain a
			regularized  version of the Hilfer derivative, we must restrict ourselves to the
			set of absolutely continuous functions $AC^1[0,b]$ and therefore applying Theorem \ref{gianduia} we obtain,
			for $\mu \in (0,1)$,
			\begin{align}
				\label{lenovo}
				D^{\mu,\nu}_{0^+}f(t)&=\left(I_{0^+}^{\nu(1-\mu)}\frac{d}{dt}%
				(I_{0^+}^{(1-\nu)(1-\mu)}f)\right)(t) \\
				&= \left( I_{0^+}^{\nu(1-\mu)}I_{0^+}^{(1-\nu)(1-\mu)}\frac{d}{dt}%
				f \right) (t) + I_{0^+}^{\nu(1-\mu)} \frac{t^{\nu\mu-\nu-\mu}f(0^+)}{%
				\Gamma(1-\nu-\mu+\nu\mu)}  \notag \\
				&=I_{0^+}^{1-\mu}\frac{d}{dt}f(t)+\frac{t^{-\mu}f(0^+)}{\Gamma(1-\mu)}
				={}^CD^{\mu}_{0^+}f(t)+\frac{t^{-\mu}f(0^+)}{\Gamma(1-\mu)},  \notag
			\end{align}
			where we used the well-known semi-group property of Riemann--Liouville
			integrals and where  ${}^CD^{\mu}_{0^+}$  is the Caputo derivative %
			\eqref{Capu}.  From \eqref{lenovo} it follows that in the space $AC^1[0,b]$ the Hilfer
			derivative \eqref{hil} coincides with the Riemann--Liouville derivative of order $\mu$,
			and the regularized Hilfer derivative can be written as
			\begin{align}
				D^{\mu,\nu}_{0^+}f(t) - \frac{t^{-\mu}f(0^+)}{\Gamma(1-\mu)},
			\end{align}
			which coincides with ${}^CD^{\mu}_{0^+}$ and which in fact does not depend on the
			parameter $\nu$.

		\section{Hilfer--Prabhakar derivatives}

			\label{barabasi}
			We introduce a  generalization of Hilfer derivatives by
			substituting in  \eqref{hil} the Riemann--Liouville  integrals with a more
			general integral operator with kernel
			\begin{align}
				e^{\gamma}_{\rho,\mu,\omega}(t) = t^{\mu-1}E^{\gamma}_{\rho,\mu}\left(\omega t^{\rho} \right),
				\qquad t \in \mathbb{R}, \: \rho, \mu, \omega, \gamma\in \mathbb{C}, \: \Re(\rho),\Re(\mu)>0, 
			\end{align}
			where
		    \begin{equation}
		        E^{\gamma}_{\rho,\mu}(x)=\sum_{k=0}^{\infty}\frac{\Gamma(\gamma+k)}{%
		        \Gamma(\gamma)\Gamma(\rho k+\mu)}\frac{x^k}{k!},
		    \end{equation}
		    is the generalized Mittag--Leffler function first investigated in %
		    \cite{Prab}.			
			The so-called Prabhakar integral is
			defined as follows \cite{Prab,Kil}.
			\begin{definition}[Prabhakar integral]
			    Let $f \in L^1[0,b]$, $0 < t < b \le \infty$. The Prabhakar integral can be written as
			    \begin{equation}
			        \label{pra}
			        \mathbf{E}^{\gamma}_{\rho,\mu, \omega, 0^+}f(t)
			        =\int_0^{t}(t-y)^{\mu-1}E^{\gamma}_{\rho,\mu}\left[\omega (t-y)^{\rho} %
			        \right]f(y)dy = (f \ast e^{\gamma}_{\rho,\mu,\omega} )(t),
			    \end{equation}
			    where $\rho, \mu, \omega, \gamma\in \mathbb{C}$, with
			    $\Re(\rho),\Re(\mu)>0$.
			\end{definition}
			We also recall that the left-inverse to the operator 
			\eqref{pra}, the Prabhakar derivative, was introduced in \cite{Kil}. We define it below in a slightly different form.
			\begin{definition}[Prabhakar derivative]
			    Let $f \in L^1[0,b]$, $0 < x < b \le \infty$, and $f \ast e_{\rho, m-\mu,\omega}^{-\gamma}(\cdot) \in W^{m,1}[0,b]$,
			    $m = \lceil \mu \rceil$.
			    The Prabhakar derivative is defined as
			    \begin{align}
			        \mathbf{D}^{\gamma}_{\rho, \mu, \omega, 0^+}f(x)
			        = \frac{d^m}{dx^m}\mathbf{E}^{-\gamma}_{\rho, m-\mu, \omega, 0^+}f(x)
			    \end{align}
			    where $\mu,\omega,\gamma,\rho \in \mathbb{C}$, $\Re(\mu),\Re(\rho)>0$.
			\end{definition}
			Observing
			now that the Riemann--Liouville integrals in \eqref{rlder} can be  expressed
			in terms of Prabhakar integrals as
			\begin{equation}
				I_{0^+}^{m-(\mu+\theta)}f(x)= \mathbf{E}^0_{\rho,
				m-(\mu+\theta),\omega, 0^+}f (x),
			\end{equation}
			we have that,
			\begin{align}
			    \label{pde1}
			    \mathbf{D}^{\gamma}_{\rho, \mu, \omega, 0^+}f(x)
			    & = \frac{d^m}{dx^m}\mathbf{E}^{-\gamma}_{\rho, m-\mu, \omega,
			    0^+}f(x)\\
			    \nonumber &= \frac{d^m}{dx^m}I^{m-(\mu+\theta)}\mathbf{E}%
			    ^{-\gamma}_{\rho, \theta, \omega, 0^+}f(x) \\
			    & = D^{\mu+\theta}_{0^+} \mathbf{E}^{-\gamma}_{\rho, \theta,
			    \omega, 0^+}f(x), \qquad \theta \in \mathbb{C}, \: \Re(\theta)>0, \notag
			\end{align}
			where we used the fact that (see \cite{Kil}, Theorem 8)
			\begin{equation}
				\label{smi}
				\mathbf{E}_{\rho, \mu,\omega, 0^+}^{\gamma}\mathbf{E}_{\rho, \nu,\omega,
				0^+}^{\sigma}f(x) =\mathbf{E}_{\rho, \mu+\nu,\omega, 0^+}^{\gamma+\sigma}
				f(x).
			\end{equation}
			Note that formula \eqref{pde1} concides with the definition given by \cite{Kil}.
			As expected, the inverse operator \eqref{pde1} of the Prabhakar  integral
			generalizes the Riemann--Liouville  derivative. Its
			regularized Caputo counterpart is given, for functions $f \in AC^m[0,b]$, $0 < x < b \le \infty$, by
			\begin{align}
				\label{capl}
				{}^C\mathbf{D}^{\gamma}_{\rho, \mu, \omega,
				0^+}f(x)&=\mathbf{E}^{-\gamma}_{\rho, m-\mu, \omega, 0^+}\frac{d^m}{dx^m}
				f(x) \\
				&=\mathbf{D}^{\gamma}_{\rho, \mu, \omega,
				0^+}f(x)- \sum_{k=0}^{m-1} x^{k-\mu} E^{-\gamma}_{\rho,k-\mu+1}(\omega x^{\rho}) f^{(k)}(0^+). \notag
			\end{align}
			\begin{prop}
			    Let $\mu>0$ and $f \in AC^m[0,b]$, $0 < x < b \le \infty$. Then
			    \begin{align}
			        {}^C\mathbf{D}^{\gamma}_{\rho, \mu, \omega,0^+}f(x) = \mathbf{D}^{\gamma}_{\rho, \mu, \omega,0^+}
			        \left[ f(x) - \sum_{k=0}^{m-1}\frac{x^k}{k!}f^{(k)}(0^+) \right].
			    \end{align}
			\end{prop}
			\begin{proof}
			    It easily follows from \eqref{pde1} and Corollary 2.3 of \cite{Kil}.
			\end{proof}

			We are now ready to define the Hilfer--Prabhakar derivative,  interpolating %
			\eqref{capl} and \eqref{pde1}.
			\begin{definition}[Hilfer--Prabhakar derivative]
				\label{ge}
			    Let $\mu\in (0,1)$, $\nu \in[0,1]$, and let $f \in L^1[a,b]$,
			    $0 < t < b \le \infty$, $f \ast e_{\rho, (1-\nu)(1-\mu),\omega}^{-\gamma(1-\nu)}(\cdot) \in AC^1[0,b]$.
				The Hilfer--Prabhakar derivative is defined by
				\begin{equation}
					\label{hilg}
					\mathcal{D}^{\gamma, \mu, \nu}_{\rho, \omega, 0^+} f(t) =\left(\mathbf{E}%
					_{\rho, \nu(1-\mu),\omega, 0^+}^{-\gamma \nu}\frac{d}{dt}( \mathbf{E}_{\rho,
					(1-\nu)(1-\mu),\omega, 0^+}^{-\gamma(1-\nu)}f)\right)(t),
				\end{equation}
				where $\gamma, \omega \in \mathbb{R}$, $\rho > 0$, and where $\mathbf{E}_{\rho, 0,\omega, 0^+}^0 f = f$.
			\end{definition}

			We observe that \eqref{hilg} reduces to the Hilfer derivative for $\gamma=0$.
			Moreover, for $\nu =1$ and $\nu = 0$ it coincides with \eqref{capl} and %
			\eqref{pde1}, respectively (note that $m=1$).
			
			\begin{lem}
				\label{l1}  The Laplace transform of \eqref{hilg} is given by
				\begin{align}
					\label{ltra}
					& \mathcal{L}\left(\mathbf{E}_{\rho, \nu(1-\mu),\omega, 0^+}^{-\gamma\nu}\frac{d%
					}{dt}( \mathbf{E}_{\rho, (1-\nu)(1-\mu),\omega, 0^+}^{-\gamma(1-\nu)}f)\right)(s) \\
					& = s^{\mu}[1-\omega s^{-\rho}]^{\gamma}\mathcal{L}[f](s)-
					s^{-\nu(1-\mu)}[1-\omega s^{-\rho}]^{\gamma\nu}\left[ \mathbf{E}%
					^{-\gamma(1-\nu)}_{\rho,(1-\nu)(1-\mu), \omega, 0^+}f (t)\right]_{t=0^+}.
					\notag
				\end{align}
			\end{lem}
			\begin{proof}
				By recurring to formula (2.19) of \cite{Kil}, i.e.
				\begin{equation}
					\label{lap}
					\mathcal{L}\left[t^{\mu-1}E^{-\gamma}_{\rho, \mu}(\omega t^{\rho})\right](s)
					=s^{-\mu}(1-\omega s^{-\rho})^{\gamma}, \qquad \gamma, \omega, \rho, \mu\in
					\mathbb{C}, \: \Re(\mu)>0,
				\end{equation}
				with $s \in \mathbb{C}$, $\Re(s)>0$, $|\omega s^{-\rho}|<1$.
				We can write
				\begin{align}
					\mathcal{L} & \left(\mathbf{E}_{\rho, \nu(1-\mu),\omega, 0^+}^{-\gamma\nu}\frac{%
					d}{dt}( \mathbf{E}_{\rho, (1-\nu)(1-\mu),\omega, 0^+}^{-\gamma(1-\nu)}f)\right)(s)
					\\
					= {} & \mathcal{L}\left[t^{\nu(1-\mu)-1}E^{-\gamma\nu}_{\rho,
					\nu(1-\mu)}(\omega t^{\rho})\right](s) \cdot \mathcal{L}\left[\frac{d}{dt} (%
					\mathbf{E}_{\rho, (1-\nu)(1-\mu),\omega, 0^+}^{-\gamma(1-\nu)}f)\right](s)  \notag
					\\
					= {} & s^{-\nu(1-\mu)}[1-\omega s^{-\rho}]^{\gamma\nu}s \, \mathcal{L}\left[
					t^{(1-\nu)(1-\mu)-1}E^{-\gamma(1-\nu)}_{\rho, (1-\nu)(1-\mu)}(\omega t^{\rho}) %
					\right] (s) \: \mathcal{L}[f](s)  \notag \\
					& -s^{-\nu(1-\mu)}[1-\omega s^{-\rho}]^{\gamma\nu} \left[ \mathbf{E}%
					^{-\gamma(1-\nu)}_{\rho,(1-\nu)(1-\mu), \omega, 0^+}f (t)\right]_{t=0^+}
					\notag \\
					= {} & s^{\mu}[1-\omega s^{-\rho}]^{\gamma}\mathcal{L}[f](s)-
					s^{-\nu(1-\mu)}[1-\omega s^{-\rho}]^{\gamma\nu}\left[ \mathbf{E}%
					^{-\gamma(1-\nu)}_{\rho,(1-\nu)(1-\mu), \omega, 0^+}f (t)\right]_{t=0^+}.
					\notag
				\end{align}
			\end{proof}

			In order to consider Cauchy problems involving initial conditions depending
			only on the function and its integer-order  derivatives we  use the
			regularized version of \eqref{hilg}, that is, for $f \in AC^1[0,b]$, we have
			\begin{align}
				\label{regn}
				{}^C\mathcal{D}^{\gamma, \mu}_{\rho, \omega, 0^+} f(t) =\left(\mathbf{E}_{\rho,
				\nu(1-\mu), \omega, 0^+}^{-\gamma\nu} \mathbf{E}_{\rho, (1-\nu)(1-\mu),\omega,
				0^+}^{-\gamma(1-\nu)} \frac{d}{dt}f \right)(t) = \left(\mathbf{E}_{\rho,
				1-\mu,\omega, 0^+}^{-\gamma} \frac{d}{dt}f\right)(t),
			\end{align}
			We remark that, in the regularized version of the Hilfer--Prabhakar
			derivative (as well as in the regularized Hilfer derivative---see %
			\eqref{lenovo}),  there is no dependence on the interpolating parameter  $\nu$.

			\begin{lem}
				\label{l2}  The Laplace transform of the operator \eqref{regn} is given by			
				\begin{align}
					\label{penna}
					\mathcal{L}[{}^C\mathcal{D}^{\gamma, \mu}_{\rho, \omega, 0^+} f](s)
					=s^{\mu}[1-\omega s^{-\rho}]^{\gamma}\mathcal{L}[f](s)- s^{\mu-1}[1-\omega
					s^{-\rho}]^{\gamma}f(0^+).
				\end{align}
			\end{lem}
			\begin{proof}
				It ensues from similar calculations to those in Lemma \ref{l1}.
			\end{proof}

			From Lemmas \ref{l1} and \ref{l2} we have that the relation between the two
			operators  \eqref{hilg} and \eqref{regn} is given by
			\begin{equation}
				{}^C\mathcal{D}^{\gamma, \mu}_{\rho, \omega, 0^+} f(t) = \mathcal{D}^{\gamma,
				\mu, \nu}_{\rho, \omega, 0^+} f(t) -
				t^{-\mu}E^{-\gamma}_{\rho,1-\mu}(\omega t^{\rho})f(0^+),
			\end{equation}
			observing that, for absolutely continuous functions $f \in AC^1[0,b]$,
			\begin{equation}
				\left[\mathbf{E}^{-\gamma(1-\nu)}_{\rho,(1-\nu)(1-\mu), \omega, 0^+}f
				(t)\right]_{t=0^+}=0,
			\end{equation}
			and
			\begin{equation}
				\mathcal{L}^{-1}[s^{\mu-1}[1-\omega s^{-\rho}]^{\gamma}](t) \: f(0^+)
				=t^{-\mu}E^{-\gamma}_{\rho,1-\mu}(\omega t^{\rho}) \: f(0^+).
			\end{equation}

		\section{Applications}

			\label{appll}
			We show below some applications of Hilfer--Prabhakar
			derivatives in equations of interest for mathematical physics and probability.

			\subsection{Time-fractional heat equation}

				In the recent years more and more papers have been devoted to the mathematical
				analysis of versions of the time-fractional heat equation and to the study of its
				applications in mathematical physics and
				probability theory (see for example \cite{pa,1,2,3,4,5} and the references therein).
								
				Here we study a generalization of the time-fractional heat equation involving
				Hilfer--Prabhakar derivatives.  We present analytical results for the
				time-fractional heat equation involving both regularized and non regularized
				Hilfer--Prabhakar derivatives in order to highlight the main differences between the
				two cases.
				
				We start by considering the fractional heat equation  involving the
				non-regularized operator $\mathcal{D}^{\gamma, \mu, \nu}_{\rho, \omega,
				0^+}$.

				\begin{te}
					\label{prim}
					The solution to the Cauchy problem
					\begin{equation}
						\label{heo}
						\begin{cases}
							\mathcal{D}_{\rho ,\omega ,0^{+}}^{\gamma ,\mu ,\nu }u(x,t)=K\frac{\partial
							^{2}}{\partial x^{2}}u(x,t), & t>0,\:x\in \mathbb{R}, \\
							\left[ \mathbf{E}_{\rho ,(1-\nu )(1-\mu ),\omega ,0^{+}}^{-\gamma(1-\nu)
							}u(x,t)\right] _{t=0^{+}}=g(x), &  \\
							\lim_{x\rightarrow \pm \infty }u(x,t)=0, &
						\end{cases}
					\end{equation}%
					with $\mu \in (0,1)$, $\nu \in [0,1]$, $\omega \in \mathbb{R}$, $K,\rho>0$, $\gamma \ge 0$, is given by
					\begin{equation}
						u(x,t)=\int_{-\infty }^{+\infty
						} dk \, e^{-ikx}\hat{g}(k)\frac{1}{2\pi }\sum_{n=0}^{\infty }\left( -K\right) ^{n}t^{\mu \left(
						n+1\right) -\nu (\mu -1)-1}E_{\rho ,\mu \left( n+1\right) -\nu (\mu
						-1)}^{\gamma \left( n+1-\nu\right) }(\omega t^{\rho })k^{2n}.
					\end{equation}
				\end{te}
				\begin{proof}
					We denote with $\tilde{u}(x,s)=\mathcal{L}(u)(x,s)$ the Laplace
					transform with respect to the time variable $t$ and $\hat{u}(k,t)=\mathcal{F}%
					(u)(k,t)$ the Fourier transform with respect to the space variable $x$.
					Taking the Fourier--Laplace transform of \eqref{heo}, by formula \eqref{ltra}, we have
					\begin{equation}
						s^{\mu }(1-\omega s^{-\rho })^{\gamma }\hat{\tilde{u}}(k,s)-s^{\nu (\mu
						-1)}(1-\omega s^{-\rho })^{\gamma \nu}\hat{g}(k)=-Kk^{2}\hat{\tilde{u}}(k,s),
					\end{equation}%
					so that
					\begin{align}
						\hat{\tilde{u}}(k,s)& =\frac{s^{\nu (\mu -1)}(1-\omega s^{-\rho })^{\gamma \nu}%
						\hat{g}(k)}{s^{\mu }(1-\omega s^{-\rho })^{\gamma
						}+Kk^{2}}\\
						\nonumber &=s^{-\mu
						+\nu (\mu -1)}(1-\omega s^{-\rho })^{-\gamma(1-\nu) }\hat{g}(k)  \left( 1+\frac{Kk^{2}}{s^{\mu }(1-\omega
						s^{-\rho })^{\gamma }} \right)^{-1} \\
						& =\sum_{n=0}^{\infty }\left( -Kk^{2}\right) ^{n}s^{-\mu \left( n+1\right)
						+\nu (\mu -1)}(1-\omega s^{-\rho })^{-\gamma \left( n+1-\nu\right) }\hat{g}%
						(k),\qquad \left\vert \frac{Kk^{2}}{s^{\mu }(1-\omega s^{-\rho })^{\gamma }}%
						\right\vert <1.  \notag
					\end{align}%
					Inverting first the Laplace transform it yields
					\begin{equation}
						\label{batata}
						\hat{u}(k,t)=\sum_{n=0}^{\infty }\left( -K\right) ^{n}t^{\mu \left(
						n+1\right) -\nu (\mu -1)-1}E_{\rho ,\mu \left( n+1\right) -\nu (\mu
						-1)}^{\gamma \left( n+1-\nu\right) }(\omega t^{\rho })k^{2n}\hat{g}(k).
					\end{equation}
					Note that the inversion term by term of the Laplace transform is possible in view of Theorem 30.1 by \citet{doetsch}
					provided to choose a sufficiently large abscissa for the inverse integral and by recalling that the generalized Mittag Leffler
					function is defined as an absolutely convergent series.
					The convergence of \eqref{batata} and in general of series of the same form
					(see below) can be proved by using the same technique as in
					Appendix C of \cite{tomtom}.
					Indeed the function \eqref{batata} is in fact a repeated series:
					\begin{align}
						\hat{u}(k,t)=\hat{g}(k)t^{ -\nu (\mu -1)-1+\mu}\sum_{n=0}^{\infty }
						\frac{\left( -K k^2\right) ^{n}t^{\mu n}}{\Gamma(\gamma(n+1-\nu))}
						\sum_{r=0}^\infty \frac{(\omega t^\rho)^r \Gamma(r+\gamma(n+1-\nu))}{r!
						\Gamma(\rho r+ \mu (n+1)-\nu(\mu-1))}.
					\end{align}
					Since the three-parameter generalized Mittag--Leffler is an entire function,
					in order to prove the absolute convergence of \eqref{batata} it is sufficient
					to show that for each $r \in \mathbb{N}\cup \{0\}$
					\begin{align}
						\label{cco}
						\sum_{n=0}^\infty \frac{(-Kk^2 t^\mu)^n \Gamma(r+\gamma(n+1-\nu))}{\Gamma(\gamma(n+1-\nu)) \Gamma(\rho r
						+\mu(n+1)-\nu(\mu-1))}
					\end{align}
					converges absolutely. We need to study the ratio
					\begin{align}
						\label{penna2}
						& \left| \frac{(-Kk^2t^\mu)^{n+1}\Gamma(r+\gamma(n+2-\nu))}{\Gamma(\gamma(n+2-\nu))
						\Gamma(\rho r+\mu(n+2)-\nu(\mu-1))} \frac{\Gamma(\gamma(n+1-\nu))
						\Gamma(\rho r+\mu(n+1)-\nu(\mu-1))}{(-Kk^2t^\mu)^n \Gamma(r+\gamma(n+1-\nu))} \right| \\
						& = \left| (-Kk^2 t^\mu) \right| \left| \frac{\Gamma(\gamma(n+1-\nu)+\gamma+r)}{
						\Gamma(\gamma(n+1-\nu)+\gamma)} \right| \left| \frac{\Gamma(\gamma(n+1-\nu))}{\Gamma(\gamma(n+1-\nu)+r)} \right|
						\left| \frac{\Gamma(\rho r+\mu(n+1)-\nu(\mu-1))}{\Gamma(\rho r+\mu(n+1)-\nu(\mu-1)+\mu)} \right| \notag \\
						& \approx \left| (-Kk^2 t^\mu) \right| \left[ \rho r+\mu(n+1)-\nu(\mu-1) \right]^{-\mu}. \notag
					\end{align}
					Last step of the above formula is valid for large values of $k$ and can be determined by means of the well-known
					asymptotics
					\begin{align}
						\frac{\Gamma(z+a)}{\Gamma(z+b)} = z^{a-b} \left[ 1+\frac{(a-b)(a+b-1)}{2z} + O(z^{-2})\right],
					\end{align}
					for $|z|\to\infty$, $|\arg(z)|\le \pi-\varepsilon$, $|\arg(z+a)|\le\pi-\varepsilon$, $0<\varepsilon< \pi$.
					Formula \eqref{penna2} is zero for $n \to \infty$ implying absolute convergence of \eqref{cco} and therefore
					of \eqref{batata}.					
					
					To conclude the proof of the theorem, by applying the inverse Fourier transform to \eqref{batata} we obtain the claimed result.
				\end{proof}

				We now discuss the case with the regularized Hilfer--Prabhakar derivative $%
				{}^C\mathcal{D}^{\gamma, \mu}_{\rho, \omega, 0^+}$.

				\begin{te}
					The solution to the Cauchy problem
					\begin{equation}
						\label{he1}
						\begin{cases}
							{}^C\mathcal{D}_{\rho ,\omega ,0^{+}}^{\gamma ,\mu }u(x,t)=K\frac{\partial ^{2}}{%
							\partial x^{2}}u(x,t), & t>0,\:x\in \mathbb{R}, \\
							u(x,0^+)=g(x), &  \\
							\lim_{x\rightarrow \pm \infty }u(x,t)=0, &
						\end{cases}
					\end{equation}%
					with $\mu \in (0,1)$, $\omega \in \mathbb{R}$, $K,\rho>0$, $\gamma\ge 0$, is given by
					\begin{equation}
						u(x,t)= \int_{-\infty }^{+\infty } dk \,  e^{-ikx}\hat{g}(k) \frac{1}{2\pi }\sum_{n=0}^{\infty }\left( -Kt^{\mu}\right)
						^{n}E_{\rho ,\mu n+1}^{\gamma n}\left( \omega t^{\rho
						}\right)k^{2n}.
					\end{equation}
				\end{te}
				\begin{proof}
					Taking the Fourier--Laplace transform of \eqref{he1}, by formula %
					\eqref{penna}, we have that
					\begin{equation}
						s^{\mu }(1-\omega s^{-\rho })^{\gamma }\hat{\tilde{u}}(k,s)-s^{\mu
						-1}(1-\omega s^{-\rho })^{\gamma }\hat{g}(k)=-Kk^{2}\hat{\tilde{u}}(k,s),
					\end{equation}%
					so that
					\begin{align}
						\hat{\tilde{u}}(k,s)& =\frac{s^{\mu -1}(1-\omega s^{-\rho })^{\gamma }\hat{g%
						}(k)}{s^{\mu }(1-\omega s^{-\rho })^{\gamma }+Kk^{2}}
						=s^{-1}\hat{g}(k)  \left( 1+\frac{Kk^{2}}{s^{\mu}(1-\omega s^{-\rho })^{\gamma }} \right)^{-1} \\
						& =\sum_{n=0}^{\infty }\left( -Kk^{2}\right) ^{n}s^{- \mu
						n-1}(1-\omega s^{-\rho })^{-\gamma n}\hat{g}(k),\qquad \left\vert \frac{%
						Kk^{2}}{s^{\mu}(1-\omega s^{-\rho })^{\gamma }}\right\vert <1.  \notag
					\end{align}%
					Inverting first the Laplace transform (see the proof of Theorem \ref{prim} for more information on the inversion) it yields
					\begin{equation}
						\hat{u}(k,t)=\sum_{n=0}^{\infty }\left( -Kt^{\mu}\right) ^{n}
						E_{\rho, \mu n+1}^{\gamma n}(\omega t^{\rho })k^{2n}\hat{g}(k).
					\end{equation}%
					By applying the inverse Fourier transform we obtain the claimed result.
				\end{proof}


			\subsection{Fractional free electron laser equation}

				The free electron laser integro-differential equation
				\begin{equation}
					\begin{cases}
						\frac{dy}{dx}=-i\pi g\int_0^x (x-t)e^{i\eta(x-t)}y(t)dt,\qquad g,\eta\in
						\mathbb{R}, \: x\in (0,1], \\
						y(0)=1,%
					\end{cases}%
				\end{equation}
				describes the unsaturated behavior of the free electron laser (FEL) (see for
				example \cite{dat}). In recent years  many attempts to solve the
				generalized fractional integro-differential FEL equation have been proposed
				(see for example  \cite{kil1}). Here we consider the following fractional
				generalization of the FEL equation, involving Hilfer--Prabhakar derivatives.
				\begin{equation}
					\label{he4}
					\begin{cases}
						\mathcal{D}^{\gamma, \mu, \nu}_{\rho, \omega, 0^+}y(x)=\lambda
						\mathbf{E}^{\varpi}_{\rho,\mu, \omega, 0^+}y(x)+f(x), & x \in (0,\infty), \: f(x) \in L^1[0,\infty), \\
						\left[ \mathbf{E}^{-\gamma(1-\nu)}_{\rho,(1-\nu)(1-\mu), \omega,
						0^+}y(x)\right]_{x=0^+}=\kappa, & \kappa \ge 0,
					\end{cases}
				\end{equation}
				where $\mu \in (0,1)$, $\nu \in [0,1]$, $\omega,\lambda \in \mathbb{C}$, $\rho>0$, $\gamma,\varpi \ge 0$.
				This generalizes the problem studied in \cite{kil1}, corresponding to
				$\nu=\gamma=0$. Here $f(x)$ is a given function.  The original FEL
				equation is then retrieved for $\gamma=0$, $\nu=0$, $\mu \to 1$, $f\equiv 0$, $%
				\lambda = -i\pi g$, $\omega= i \eta$, $\rho=\varpi=\kappa=1$.
				
				We have the following
				
				\begin{te}
					The solution to the Cauchy problem \eqref{he4} is given by
					\begin{equation}
						y(x)=\kappa \sum_{k=0}^\infty \lambda^k x^{\nu(1-\mu)+\mu+2\mu k-1} E_{\rho,\nu(1-\mu)+\mu+2k\mu}^{\gamma+k(\varpi +\gamma)-\gamma
						\nu}(\omega x^\rho) +
						\sum_{k=0}^{\infty }\lambda^{k} \mathbf{E}_{\rho ,\mu
						(2k+1),\omega ,0^{+}}^{\gamma +k(\varpi+\gamma)}f (x).
					\end{equation}
				\end{te}
				\begin{proof}
					By taking the Laplace transform of \eqref{he4} (see \eqref{ltra}) we get
					\begin{equation}
						s^{\mu }(1-\omega s^{-\rho })^{\gamma } \mathcal{L}[y](s) - \kappa s^{-\nu(1-\mu)}(1-\omega s^{-\rho})^{\gamma \nu}
						=\lambda \mathcal{L}[x^{\mu -1}E_{\rho ,\mu }^{\varpi }(\omega
						x^{\rho })] (s)\cdot \mathcal{L}[y] (s)+ \mathcal{%
						L}[f] (s),
					\end{equation}%
					so that
					\begin{align}
						\mathcal{L}[y] (s) = {} & \frac{\kappa s^{-\nu(1-\mu)-\mu}(1-\omega s^{-\rho})^{\gamma \nu -\gamma}}{
						1-\lambda s^{-2\mu}(1-\omega s^{-\rho})^{-\varpi-\gamma}}
						+ \frac{s^{-\mu}(1-\omega s^{-\rho})^{-\gamma} \mathcal{L}[f](s)}{1-\lambda s^{-2\mu}(1-\omega s^{-\rho})^{-\varpi-\gamma}} \\
						= {} & \kappa \sum_{k=0}^\infty \lambda^k s^{-\nu(1-\mu)-\mu-2\mu k} (1-\omega s^{-\rho})^{\gamma \nu-\gamma-k(\varpi+\gamma)}
						\notag \\
						& + \sum_{k=0}^\infty \lambda^k s^{-\mu(2k +1)}(1-\omega s^{-\rho})^{-\gamma-k(\varpi+\gamma)} \mathcal{L}[f](s). \notag	
					\end{align}%
					Last step is valid for $|\lambda s^{-2\mu} (1-\omega s^{-\rho})^{-\varpi -\gamma}|$.
					Inverting the Laplace transform (see the proof of Theorem 5.1 for more information) and using the convolution theorem, we obtain
					the claimed result.
				\end{proof}

				\begin{ex}
					Let us consider the Cauchy problem \eqref{he4} with $\kappa=0$, $f(x)=x^{\mathfrak{m} -1}$. By
					direct calculation we have that
					\begin{equation}
						\mathbf{E}_{\rho ,\mu (2k+1),\omega ,0^{+}}^{\gamma +k(\varpi+\gamma)}x^{\mathfrak{m}
						-1}=\Gamma (\mathfrak{m} )x^{\mu (2k+1)+\mathfrak{m} -1}E_{\rho ,\mu (2k+1)+\mathfrak{m} }^{\gamma +k(\varpi+\gamma)}
						(\omega x^{\rho }),
					\end{equation}%
					and the explicit solution of the Cauchy problem is given by
					\begin{equation}
						y(x)=\Gamma (\mathfrak{m} )x^{\mu +\mathfrak{m} -1}\sum_{k=0}^{\infty }(\lambda x^{2\mu
						})^{k}E_{\rho ,\mu (2k+1)+\mathfrak{m} }^{\gamma +k(\varpi+\gamma)}
						(\omega x^{\rho }).
					\end{equation}
				\end{ex}

				\begin{ex}
					Let us consider the Cauchy problem \eqref{he4} with $\kappa=0$, $f(x)=x^{\mathfrak{m} -1}E_{\rho
					,\mathfrak{m} }^{\sigma }(\omega x^{\rho })$. Recalling that
					\begin{equation}
						\mathbf{E}_{\rho ,\mu (2k+1),\omega ,0^{+}}^{\gamma +k(\varpi+\gamma)}x^{\mathfrak{m}
						-1}E_{\rho ,\mathfrak{m} }^{\sigma }(\omega x^{\rho })=x^{\mu (2k+1)+\mathfrak{m} -1}E_{\rho
						,\mu (2k+1)+\mathfrak{m} }^{\gamma +k(\varpi+\gamma) +\sigma }(\omega x^{\rho }),
					\end{equation}%
					we obtain
					\begin{equation}
						y(x)=x^{\mu +\mathfrak{m} -1}\sum_{k=0}^{\infty }(\lambda x^{2\mu })^{k}E_{\rho ,\mu
						(2k+1)+\mathfrak{m} }^{\gamma +k(\varpi+\gamma) +\sigma }(\omega x^{\rho }).
					\end{equation}
				\end{ex}

			\subsection{Fractional Poisson processes involving Hilfer--Prabhakar derivatives}
			
				\label{poipoi}
				In this section we present a generalization of the homogeneous
				Poisson process for which the governing difference-differential equations
				contain the regularized Hilfer--Prabhakar differential operator acting in
				time. The considered framework generalizes also the time-fractional
				Poisson process which in the recent years has become subject of intense research.
				It is well known that the state probabilities of the classical Poisson
				process and its time-fractional generalization can be found by solving an
				infinite system of difference-differential equations.
				We solve an analogous infinite system and find the corresponding  state
				probabilities that we give in form of an infinite series and in integral
				form. As the zero state probability  of a renewal process coincides with
				the residual time probability, we can characterize our process  also by its
				waiting distribution (the common way of characterizing a renewal process).
				We will see in the following that the state probabilities of the generalized
				Poisson process are expressed by functions which in fact generalize
				the classical Mittag--Leffler function. 
				The Mittag--Leffler function appeared as residual waiting time between events in
				renewal processes already in the Sixties of the past century, namely
				processes with properly scaled thinning out the sequence of events
				in a power law renewal process (see \cite{gne} and \cite{g3}). Such
				a process in essence is a fractional Poisson process. It must
				however be said that Gnedenko and Kovalenko did their analysis only
				in the Laplace domain, not recognizing their result as the Laplace
				transform of a Mittag--Leffler type function. Balakrishnan in 1985
				\cite{bal} also found this Laplace transform as highly relevant for
				analysis of time-fractional diffusion processes, but did not
				identify it as arising from a Mittag--Leffler type function.  In the
				Nineties of the past century the Mittag--Leffler function arrived at
				its deserved honour, more and more researchers became aware of it
				and used it. Let us only sketch a few highlights. Hilfer and Anton
				\cite{hill2} were the first authors who explicitly introduced
				the Mittag--Leffler waiting-time density
				\begin{align}
					f_{\mu}(t)=-\frac{d}{dt}E_{\mu}(-t^{\mu}) = t^{\mu-1} E_{\mu,\mu} (-t^\mu)
				\end{align}
				(writing it in form of a
				Mittag--Leffler function with two indices) into the theory of
				continuous time random walk. They showed that it is needed if one
				wants to get as evolution equation for the sojourn density the
				fractional variant of the Kolmogorov--Feller equation. In modern
				terminology they subordinated a random walk to the fractional
				Poisson process. By completely different argumentation the authors
				of \cite{sca} also discussed the relevance of $f_{\mu}(t)$ in theory
				of continuous time random walk. However, all these authors did not
				treat the fractional Poisson process as
				a subject of study in its  own right but simply as useful for
				general analysis of certain stochastic processes. The detailed
				analytic and probabilistic investigation was started (as far as we
				know) in 2000  by Repin and Saichev \cite{Repin}. More and more
				researchers then, often independently of each other, investigated
				this renewal process. Let us here only recall the few relevant papers
				\cite{g1,g2, seb, Beghin,Macci, cah, Laskin1, fed, scalas, meer} and see also the
				references cited therein.

				Let us thus start with the governing equations for the state probabilities.
				In view of Section \ref{barabasi} we define the following
				Cauchy problem involving  the regularized operator ${}^C\mathcal{D}^{\gamma,
				\mu}_{\rho, \omega, 0^+}$.
				
				\begin{definition}[Cauchy problem for the generalized fractional Poisson process]
					\begin{align}
						\label{aa1}
						\begin{cases}
							{}^C\mathcal{D}^{\gamma, \mu}_{\rho, -\phi, 0^+} p_k(t) = -\lambda
							p_k(t) +\lambda p_{k-1}(t), & k \ge 0, \: t > 0,
							\: \lambda > 0, \\
							p_k(0) =
							\begin{cases}
								1, \quad k=0, \\
								0, \quad k \geq 1,%
							\end{cases}
							&
						\end{cases}%
					\end{align}
					where $\phi > 0$, $\gamma \ge 0$, $0 < \rho \le 1$, $0<\mu \le 1$.
					We also have $0 < \mu \lceil \gamma \rceil/\gamma - r\rho < 1$, $\forall \: r=0,\dots,\lceil \gamma \rceil$, if $\gamma \ne 0$.
				\end{definition}
				These ranges for the parameters are needed to ensure non-negativity of the solution (see Section \ref{ren}
				for more details).
				Multiplying both the terms of \eqref{aa1} by $v^k$ and adding over
				all $k$, we obtain the fractional Cauchy problem for the probability
				generating function $G(v,t) = \sum_{k=0}^\infty v^k p_k(t)$ of the counting number
				$N(t)$, $t \ge 0$,
				\begin{equation}
					\label{gen}
					\begin{cases}
						{}^C\mathcal{D}^{\gamma, \mu}_{\rho, -\phi, 0^+}G(v,t)= -\lambda (1-v)G(v,t), &
						|v| \le 1, \\
						G(v,0)=1. &
					\end{cases}%
				\end{equation}

				\begin{te}
					\label{gianni}
					The solution to \eqref{gen} reads
					\begin{equation}
						\label{G}
						G(v,t)=\sum_{k=0}^{\infty}(-\lambda t^{\mu})^k(1-v)^k E_{\rho, \mu
						k+1}^{\gamma k}(-\phi t^{\rho}), \qquad |v| \le 1.
					\end{equation}
				\end{te}
				\begin{proof}
					In view of Lemma \ref{l2},  we have
					\begin{align}
						s^{\mu}[1+\phi s^{-\rho}]^{\gamma}\mathcal{L}[G](v,s)-s^{\mu-1} [1+\phi
						s^{-\rho}]^{\gamma}=-\lambda(1-v)\mathcal{L}[G](v,s),
					\end{align}
					so that
					\begin{align}
						\label{g1}
						\mathcal{L}[G](v,s)&=\frac{s^{\mu-1}[1+\phi s^{-\rho}]^{\gamma}}{%
						s^{\mu}[1+\phi s^{-\rho}]^{\gamma}+\lambda(1-v)}
						=\frac{1}{s} \left( 1+\frac{\lambda(1-v)}{s^{\mu}[1+\phi s^{-\rho}]^{\gamma}} \right)^{-1} \\
						&= \frac{1}{s}\sum_{k=0}^{\infty}\left[-\frac{\lambda(1-v)}{s^{\mu}
						[1+\phi s^{-\rho}]^{\gamma}}\right]^k =
						\sum_{k=0}^{\infty}(-\lambda(1-v))^k s^{-\mu k-1}[1+\phi
						s^{-\rho}]^{-k\gamma},  \notag
					\end{align}
					where $|\lambda(1-v)/[s^\mu(1+\phi s^{-\rho})^{\gamma}]|<1$.  By using %
					\eqref{lap} we can invert the Laplace transform \eqref{g1} obtaining the
					claimed result (for more details on the inversion and on the convergence of series of similar form of \eqref{G} 
					the reader can consult the proof of Theorem \ref{prim}).
				\end{proof}

				\begin{os}
					Observe that for $\gamma=0$, we retrieve the classical result obtained for
					example in \cite{Laskin}, formula (23).  Indeed, from the fact that
					\begin{equation}
						\label{tomto}
						E_{\rho, \mu k+1}^{0}(-\phi t^{\rho})
						= \sum_{r=0}^\infty \frac{(-\phi t^\rho)^r \Gamma(r)}{r!\Gamma(\rho r + \mu k +1)\Gamma(0)}=\frac{1}{\Gamma(\mu k +1)},
					\end{equation}
					(note that in the series representation of the three-parameter Mittag--Leffler function \eqref{tomto}
					each term is zero except that for $r=0$) equation \eqref{G} becomes
					\begin{equation}
						G(v,t)=\sum_{k=0}^{\infty}\frac{(-\lambda t^{\mu})^k(1-v)^k}{\Gamma(\mu k+1)}
						= E^1_{\mu,1}(-\lambda(1-v)t^{\mu}),
					\end{equation}
					that coincides with equation (23) in \cite{Laskin}.
				\end{os}

				From the probability generating function \eqref{G}, we are now able to find
				the probability distribution at fixed time $t$  of $N(t)$, $t \ge 0$,
				governed by  \eqref{aa1}. Indeed, a simple binomial expansion leads to
				\begin{align}
					G(v,t) = \sum_{k =0}^{\infty}v^k \sum_{r=k}^{\infty}(-1)^{r-k}\binom{r}{k}%
					(\lambda t^{\mu})^r E_{\rho, \mu r+1}^{\gamma r}(-\phi t^{\rho}).
				\end{align}
				Therefore,
				\begin{equation}
					\label{distro}
					p_k(t)=\sum_{r=k}^{\infty}(-1)^{r-k}\binom{r}{k}(\lambda t^{\mu})^r E_{\rho,
					\mu r+1}^{\gamma r}(-\phi t^{\rho}), \qquad k \ge 0, \: t \ge 0.
				\end{equation}
				We observe that, for $\gamma =0$,
				\begin{align}
					\label{indent}
					p_k(t) & =\sum_{r=k}^{\infty}(-1)^{r-k}\binom{r}{k}\frac{(\lambda t^{\mu})^r}{
					\Gamma(\mu r+1)}
					= (\lambda t^\mu)^k E_{\mu, \mu k+1}^{k+1}(-\lambda t^\mu) \\
					& = \frac{(\lambda t^\mu)^k}{k!} E^{(k)}_{\mu,1}(-\lambda t^\mu), \qquad k\geq 0, \: t \ge 0, \notag
				\end{align}
				The first expression of \eqref{indent} coincides with equation (1.4)
				in \cite{Beghin}. The third one is a convenient representation
				involving the $k$th derivative of the two-parameter Mittag--Leffler
				function evaluated at $-\lambda t^\mu$.
				It is
				immediate to note, from \eqref{G}, by inserting $v=1$, that $\sum_{k=0}^\infty p_k(t) = 1$.
				From the generating function \eqref{G} much information on the behavior of the process can be extracted.
				From \eqref{aa1}, with standard methods
				we can evaluate
				the mean value of $N(t)$.  In order to do so, it suffices to differentiate equation \eqref{gen}
				with respect to $v$ and to take $v=1$. We obtain
				\begin{equation}
					\begin{cases}
						{}^C\mathcal{D}^{\gamma, \mu}_{\rho, -\phi, 0^+}\mathbb{E}N(t)=\lambda, & t > 0,
						\\
						\mathbb{E}N(t)\big|_{t=0}=0, &
					\end{cases}%
				\end{equation}
				whose solution is simply given by
				\begin{equation}
					\label{mean}
					\mathbb{E}N(t)=\lambda t^{\mu}E^{\gamma}_{\rho, 1+\mu}(-\phi t^{\rho}),
					\qquad t \ge 0.
				\end{equation}

				\subsubsection{Subordination representation}

					\label{sr}
					In order to derive an alternative representation for the fractional Poisson
					process $N(t)$, $t \ge 0$, we present first some  preliminaries.  Consider
					the Cauchy problem
					\begin{align}
						\label{lanwan}
						\begin{cases}
							{}^C\mathcal{D}^{\gamma, \mu}_{\rho, -\phi, 0^+} h(x,t) = - \frac{\partial}{%
							\partial x} h(x,t), & t > 0, \: x \ge 0, \\
							h(x,0^+) = \delta(x). &
						\end{cases}%
					\end{align}
					Representation \eqref{capl} and the results obtained in \cite{Pol} simply
					imply that the Laplace--Laplace transform of $h(x,t)$  can be written as
					\begin{align}
						\label{ll}
						\tilde{\tilde{h}} (z,s) = \frac{s^{\mu-1}(1+\phi s^{-\rho})^{\gamma}}{%
						s^\mu (1+\phi s^{-\rho})^{\gamma}+z}, \qquad s > 0, \: z > 0.
					\end{align}
					This can be easily seen by taking the Laplace transform of \eqref{lanwan} with respect to both variables
					$x$ and $t$ and by using Lemma \ref{l2}. We have
					\begin{align}
						s^\mu (1+\phi s^{-\rho})^\gamma \tilde{\tilde{h}}(z,s) -s^{\mu-1}(1+\phi s^{-\rho})^\gamma =
						- z \, \tilde{\tilde{h}}(z,s),
					\end{align}
					which immediately leads to \eqref{ll}.
					Consider now the stochastic process, given as a finite sum of subordinated
					independent subordinators
					\begin{align}
						\mathfrak{V}_t = \sum_{r=0}^{\lceil \gamma \rceil} {}_r V_{\Phi(t)}^{\mu
						\frac{\lceil \gamma \rceil}{\gamma}-r\rho}, \qquad t \ge 0.
					\end{align}
					In the above definition $\lceil \gamma \rceil$ represents the ceiling of $%
					\gamma$. Furthermore  we considered a sum of $\lceil \gamma \rceil$
					independent stable subordinators of different indices and  the random time
					change here is defined by
					\begin{align}
						\Phi(t) = \binom{\lceil \gamma \rceil}{r} V_t^{\frac{\gamma}{\lceil
						\gamma \rceil}}, \qquad t \ge 0,
					\end{align}
					where $V_t^{\frac{\gamma}{\lceil \gamma \rceil}}$ is a further stable
					subordinator, independent of the others. Note that  in order the above
					process $\mathfrak{V}_t$, $t \ge 0$, to be well-defined, the constraint  $0<
					\mu \lceil \gamma \rceil/\gamma-r\rho <1$ holds for each $%
					r=0,1,\dots,\lceil \gamma \rceil$.  The next step is to define its hitting
					time. This can be done as
					\begin{align}
						\label{inverse}
						\mathfrak{E}_t = \inf \{ s \ge 0 \colon \mathfrak{V}_s > t \}, \qquad t \ge	0.
					\end{align}
					Theorem 2.2 of \cite{Pol} ensures us that the law $\Pr \{ \mathfrak{E}_t
					\in dx \}/dx$ is the solution to the Cauchy problem  \eqref{lanwan} and
					therefore that its Laplace--Laplace transform is exactly that in \eqref{ll}.
					
					We are now ready to state the following theorem.

					\begin{te}
						Let $\mathfrak{E}_t$, $t \ge 0$, be the hitting-time process presented in
						formula \eqref{inverse}.  Furthermore let $\mathcal{N}(t)$, $t \ge 0$, be a
						homogeneous Poisson process of parameter $\lambda>0$, independent of  $%
						\mathfrak{E}_t$.  The equality
						\begin{align}
							N(t) = \mathcal{N}(\mathfrak{E}_t), \qquad t \ge 0,
						\end{align}
						holds in distribution.
					\end{te}
					\begin{proof}
						The claimed relation can be proved simply writing the probability generating
						function related to the time-changed  process $\mathcal{N}(\mathfrak{E}_t)$
						as
						\begin{align}
							\sum_{k=0}^\infty v^k \Pr(\mathcal{N}(\mathfrak{E}_t)=k) = \int_0^\infty
							e^{-\lambda(1-v)y} \Pr (\mathfrak{E}_t \in dy).
						\end{align}
						Therefore, by taking the Laplace transform with respect to time we have
						\begin{align}
							\int_0^\infty \int_0^\infty e^{-\lambda(1-v)y -st} \Pr (\mathfrak{E}_t \in
							dy) dt = \frac{s^{\mu-1}(1+\phi s^{-\rho})^{\gamma}}{s^\mu (1+\phi
							s^{-\rho})^{\gamma}+\lambda (1-v)}.
						\end{align}
						Considering now the relation (2.20) of \cite{Pol}, the above Laplace
						transform can be inverted immediately, obtaining
						\begin{align}
							\sum_{k=0}^\infty v^k \Pr(\mathcal{N}(\mathfrak{E}_t)=k) = \sum_{k=0}^\infty
							(-\lambda(1-v))^k t^{\mu k} E_{\rho, k\mu+1}^{k \gamma} (-\phi t^\rho),
						\end{align}
						which coincides with \eqref{G}.
					\end{proof}

				\subsubsection{Renewal structure}

					\label{ren}
					The generalized fractional Poisson process $N(t)$, $t \ge 0$, can be
					constructed as a renewal process with  specific waiting times. Consider $k$
					i.i.d.\ random variables $T_j$, $j=1,\dots,k$, representing the  inter-event
					waiting times and having probability  density function
					\begin{align}
						\label{tempi}
						f_{T_j}(t_j) = \lambda t_j^{\mu-1} \sum_{r=0}^\infty (-\lambda
						t_j^\mu)^r E_{\rho,\mu r+\mu}^{\gamma r+\gamma} (-\phi t_j^\rho),
						\qquad t \ge 0, \: \mu \in (0,1),
					\end{align}
					and
					Laplace transform (recall formula (2.19) of \cite{Kil})
					\begin{align}
						\label{cornetto}
						\mathbb{E} \exp(-sT_j) & = \lambda \sum_{r=0}^\infty (-\lambda)^r s^{-\mu
						r-\mu} (1+\phi s^{-\rho})^{-\gamma r-\gamma} \\
						& = \frac{\lambda s^{-\mu}(1+\phi s^{-\rho})^{-\gamma}}{1+\lambda
						s^{-\mu}(1+\phi s^{-\rho})^{- \gamma}}, \qquad \left| -\lambda
						s^{-\mu}(1+\phi s^{-\rho})^{-\gamma} \right| < 1  \notag \\
						& = \frac{\lambda}{s^\mu (1+\phi s^{-\rho})^{\gamma}+\lambda}.  \notag
					\end{align}
					Denote $\mathcal{T}_m=T_1+T_2+\dots+T_m$ as the waiting time of the $m$th
					renewal event.  The probability distribution $\Pr(N(t)=k)$ can be written
					making the renewal structure explicit. Indeed,  by applying the Laplace
					transform to \eqref{distro} we have
					\begin{align}
						\label{bicic}
						\mathcal{L}[p_k](s) & = \sum_{r=k}^\infty (-1)^{r-k} \binom{r}{k} \lambda^r
						s^{-\mu r-1} (1+\phi s^{-\rho})^{- \gamma r} \\
						& = s^{-1} \sum_{r=0}^\infty (-1)^r \binom{r+k}{k} \left( \frac{\lambda}{%
						s^\mu (1+\phi s^{-\rho})^{\gamma}} \right)^{r+k}  \notag \\
						& = s^{-1} \lambda^k s^{-\mu k}(1+\phi s^{-\rho})^{-\gamma k}
						\sum_{r=0}^\infty \binom{-k-1}{r} \left( \frac{\lambda}{s^\mu (1+\phi
						s^{-\rho})^{\gamma}} \right)^r  \notag \\
						& = s^{-1} \lambda^k s^{-\mu k}(1+\phi s^{-\rho})^{-\gamma k} \left( 1+%
						\frac{\lambda}{s^\mu(1+\phi s^{-\rho})^{\gamma}} \right)^{-k-1}  \notag \\
						& = \frac{\lambda^k s^{\mu-1} (1+\phi s^{-\rho})^{\gamma}}{%
						[s^\mu(1+\phi s^{-\rho})^{\gamma}+\lambda]^{k+1}}.  \notag
					\end{align}
					On the other hand, by exploiting the renewal structure,
					\begin{align}
						\label{calculations}
						\mathcal{L}[p_k](s) & = \int_0^\infty e^{-st} \left( \Pr(\mathcal{T}_k<t) -
						\Pr(\mathcal{T}_{k+1}<t) \right) dt \\
						& = \int_0^\infty e^{-st} \left[ \int_0^t \Pr(\mathcal{T}_k \in dy) -
						\int_0^t \Pr(\mathcal{T}_{k+1}\in dy) \right] dt  \notag \\
						& = \int_0^\infty \Pr (\mathcal{T}_k \in dy) \int_y^\infty e^{-st} dt
						-\int_0^\infty \Pr(\mathcal{T}_{k+1} \in dy) \int_y^\infty e^{-st} dt  \notag
						\\
						& = s^{-1} \left[ \int_0^\infty e^{-sy} \Pr (\mathcal{T}_k \in dy) -
						\int_0^\infty e^{-sy} \Pr (\mathcal{T}_{k+1}\in dy)\right]  \notag \\
						& = s^{-1} \left[ \left( \frac{\lambda}{s^\mu (1+\phi
						s^{-\rho})^{\gamma}+\lambda} \right)^k - \left( \frac{\lambda}{s^\mu
						(1+\phi s^{-\rho})^{\gamma}+\lambda} \right)^{k+1} \right]  \notag \\
						& = s^{-1} \left[ \frac{\lambda^k[s^\mu(1+\phi
						s^{-\rho})^{\gamma}+\lambda]-\lambda^{k+1}}{ [s^\mu(1+\phi
						s^{-\rho})^{\gamma}+\lambda]^{k+1}} \right]  \notag \\
						& = \frac{\lambda^k s^{\mu-1} (1+\phi s^{-\rho})^{\gamma}}{%
						[s^\mu(1+\phi s^{-\rho})^{\gamma}+\lambda]^{k+1}}, \notag
					\end{align}
					which coincides with \eqref{bicic}.

					Clearly, considering the renewal structure of the process, we can write the probability
					of the residual waiting time as
					\begin{align}
					    \mathbb{P}(T_1>t) = p_0(t) = \sum_{r=0}^\infty (-\lambda t^\mu)^r E_{\rho,\mu r +1}^{\gamma r} (-\phi t^\rho).
					\end{align}

					In order to prove the non-negativity of \eqref{tempi} (and therefore of $p_k(t)$---see also the calculations in
					\eqref{calculations}) we can proceed as
					follows. We will consider only the case $\gamma \ne 0$ as the case $\gamma=0$ corresponds
					in fact to the case studied in \cite{Laskin,g4,Beghin} and others. From the
					Bernstein theorem (see e.g.\ \cite{schilling}, Theorem 1.4) it suffices to study the complete
					monotonicity of the Laplace transform \eqref{cornetto}. Recall that the function
					$z\to 1/(z+\lambda)$ is completely monotone for any positive $\lambda$ and that
					$1/(g(z)+\lambda)$ is completely monotone if $g(z)$ is a Bernstein function.
					Thus it is just a matter of proving that the function
					\begin{align}
						s^\mu(1+\phi s^{-\rho})^{\gamma} = \left( s^{\mu/\gamma} +\phi s^{\mu/\gamma - \rho} \right)^{\gamma}
					\end{align}
					is a Bernstein function.
					We have
					\begin{align}
						\label{bfunct}
						\left( s^{\mu/\gamma} +\phi s^{\mu/\gamma - \rho} \right)^{\gamma}
						&= \left[ \left( s^{\mu/\gamma}
						+\phi s^{\mu/\gamma-\rho} \right)^{\lceil \gamma \rceil} \right]^{\gamma/\lceil \gamma \rceil} \\
						& = \left( \sum_{r=0}^{\lceil \gamma \rceil} \binom{\lceil \gamma \rceil}{r}
						\phi^r s^{\mu\lceil \gamma \rceil/\gamma -\rho r} \right)^{\gamma/\lceil \gamma \rceil}. \notag
					\end{align}					
					From the fact that the space of Bernstein functions is closed under composition
					and linear combinations (see \cite{schilling} for details) we have that \eqref{bfunct}
					is a Bernstein function for $0 < \mu \lceil \gamma \rceil/\gamma - r\rho < 1$, $\forall \: r=0,\dots,\lceil \gamma \rceil$,
					which coincide with the constraints derived in Section \ref{sr}.
					
					Table \ref{ta} shows the relevant formulas for the generalized fractional Poisson process along with those
					of the classical time-fractional Poisson process.
					
					\begin{table}\centering
						\begin{tabular}{lll}\toprule
							& fPp & $E_\mu(-\lambda(1-v) t^\mu)$ \\
							$G(v,t)$ \phantom{abd} & & \\
							& gfPp \phantom{abd} & $\sum_{k=0}^{\infty}(-\lambda t^{\mu})^k(1-v)^k E_{\rho, \mu
							k+1}^{\gamma k}(-\phi t^{\rho})$ \\
							\\ \\
							& fPp & $(\lambda t^\mu)^k E_{\mu, \mu k+1}^{k+1}(-\lambda t^\mu)$ \\
							$p_k(t)$ & & \\
							& gfPp & $\sum_{r=k}^{\infty}(-1)^{r-k}\binom{r}{k}(\lambda t^{\mu})^r E_{\rho,
							\mu r+1}^{\gamma r}(-\phi t^{\rho})$ \\
							\\ \\
							& fPp & $\lambda t^{\mu-1} E_{\mu,\mu}(-\lambda t^\mu)$ \\
							$f_T(t)$ & & \\
							& gfPp & $\lambda t^{\mu-1} \sum_{r=0}^\infty (-\lambda t^\mu)^r
							E_{\rho,\mu r+\mu}^{\gamma r+\gamma} (-\phi t^\rho)$ \\
							\bottomrule
						\end{tabular}
						\caption{\label{ta}The probability generating function $G(v,t)$, the state probabilities $p_k(t)$ and
						the inter-arrival probability density function $f_T(t)$ for both the time-fractional Poisson process (fPp) and the
						generalized fractional Poisson process (gfPp).}
					\end{table}

				\subsubsection{Fractional integral of $N(t)$}

					In the recent paper \cite{fed1}, the authors considered the  Riemann--Liouville fractional
					integral
					\begin{equation}
						\mathcal{N}^{\alpha, \mu}(t)=\frac{1}{\Gamma(\alpha)}\int_0^t(t-s)^{%
						\alpha-1}\mathcal{N}^{\mu}(s)ds, \qquad \alpha>0,
					\end{equation}
					where $\mathcal{N}^{\mu}(t)$, $t\geq 0$, is the time-fractional Poisson process, whose
					state-probabilities are governed by
					difference-differential equations involving Caputo derivatives.
					They discussed some relevant characteristics of the obtained process $\mathcal{N}%
					^{\alpha, \mu}(t)$ such as its mean and variance.  Following this idea, we
					consider the fractional integral of $%
					N(t)$, $t \ge  0$. In particular,
					\begin{equation}
						N^{\alpha,\mu}(t)=\frac{1}{\Gamma(\alpha)}\int_0^t(t-s)^{\alpha-1}N(s)ds,
						\quad \mu\in(0,1), \: \alpha>0, \: t> 0.
					\end{equation}
					We can explicitly calculate the mean of the process  $N^{\alpha,\mu}(t)$ by using \eqref{mean}.
					\begin{align}
						\mathbb{E} N^{\alpha,\mu}(t)&=\frac{1}{\Gamma(\alpha)}\int_0^t(t-s)^{\alpha-1}
						\mathbb{E}N(s) ds \\
						&=\frac{\lambda}{\Gamma(\alpha)}\int_0^t(t-s)^{\alpha-1} s^{\mu}E^{\gamma}_{\rho,
						1+\mu}(-\phi s^{\rho})ds  \notag \\
						& = \lambda \int_0^t (t-s)^{\alpha-1} E_{\rho,\alpha}^0 (-\phi(t-s)^\rho) s^\mu E_{\rho,\mu+1}^{\gamma} (-\phi s^\rho)
						ds \notag \\
						&= \lambda t^{\alpha+\mu}E^{\gamma}_{\rho, \alpha+\mu+1}(-\phi
						t^{\rho}), \qquad t \ge 0.  \notag
					\end{align}

					Note also that, exploiting Theorem 2 of \cite{Kil}, a more general Prabhakar
					integral of the process $N(t)$, $t \ge 0$, can be studied.
					
			\appendix
				
			\section{Comments on the proofs of Theorems \ref{prim} and \ref{gianni}}
			
				For each $n=0,\dots,\infty$,
				the Laplace transforms in the last line of formula (33)
				can be inverted by choosing an appropriated
				Bromwich contour (so that all the singularities lie to the left of the path). In particular the abscissa $c$
				of the vertical line on which the integral is computed must be chosen consistently with the constraints
				$\Re(s) > 0$, $|\omega s^{-\rho}|<1$, and
				$|Kk^2/(s^\mu(1-\omega s^{-\rho})^\gamma)|<1$ (see the figure below for an example of the latter constraint).
				\begin{figure}[ht!]
					\centering
					\includegraphics[scale=0.7]{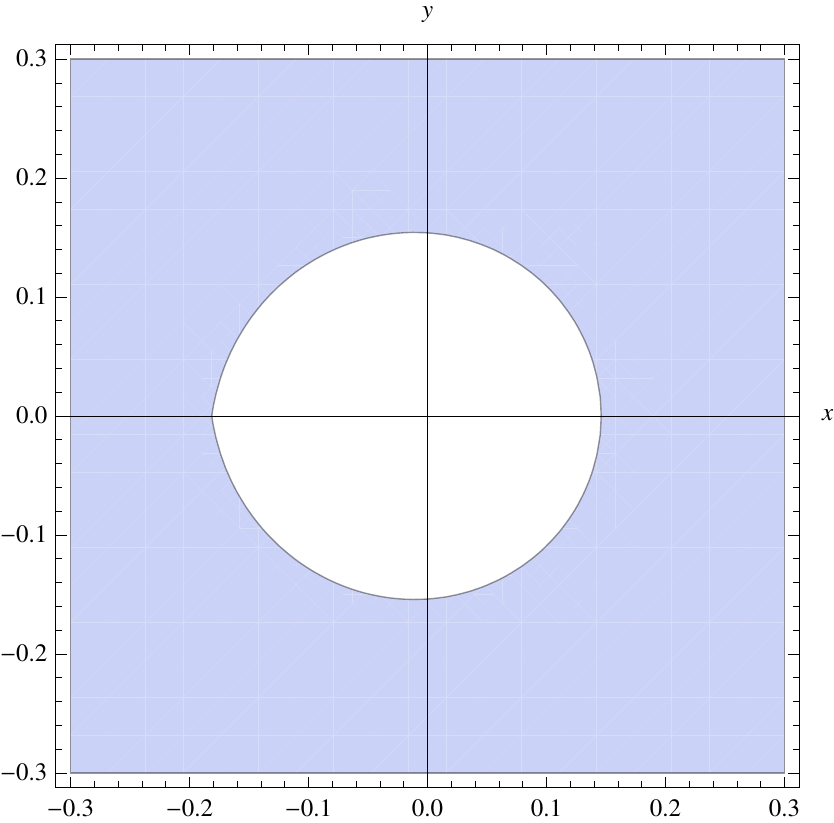}
					\caption{The constraint $|Kk^2/(s^\mu(1-\omega s^{-\rho})^\gamma)|<1$ (blue region,
						with $x=\Re(s)$, $y=\Im(s)$) for $(K,k,\mu,\omega,\rho,\gamma)
						=(1,1,0.5,-1,0.25,1)$.}
				\end{figure}
			The abscissa $c$ 
			clearly depends on $k$ which varies in $\mathbb{R}$ (notice however that $c$ does not depend on $n$). The
			inversion term by term of the Laplace transform is then permitted by Theorem 30.1 of \citet{doetsch} and
			the series of the inverse transforms converges absolutely in $t \ge 0$
			irrespective of the value of $c$ chosen (provided it is sufficently large).
		
			In Fig.~2 it is possible to actually see the region of validity of the constraint
			$|\lambda (1-v)/(s^\mu(1+\phi s^{-\rho})^\gamma)|<1$ for a specific choice of the parameters. As before,
			the application of the inverse Laplace transform term by term is ensured by Theorem 30.1
			of \citet{doetsch}. 
			\begin{figure}[ht!]
				\centering
				\includegraphics[scale=0.7]{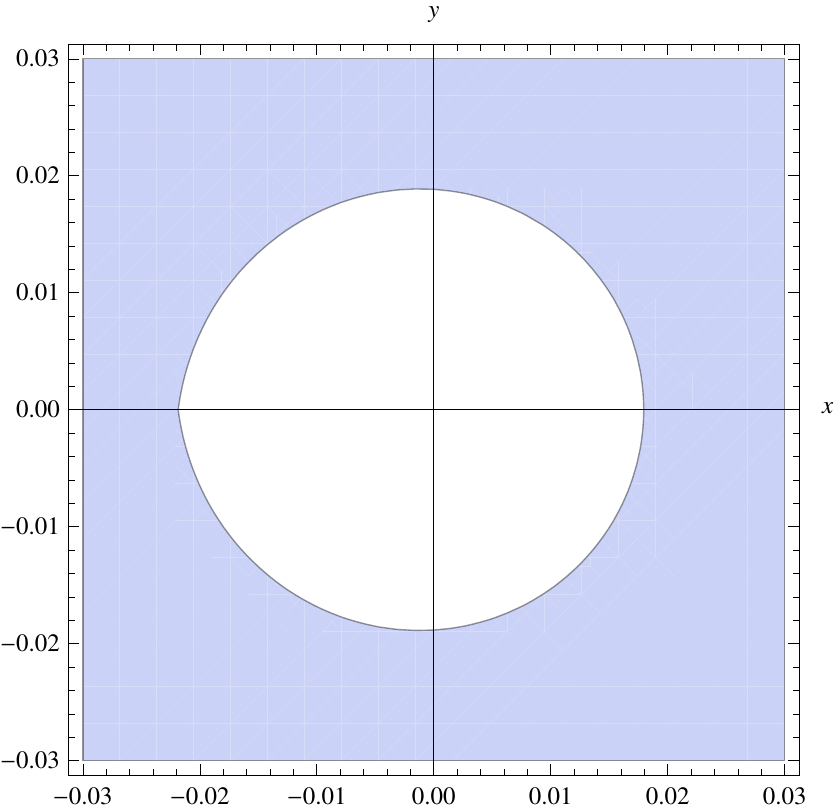}
				\caption{The constraint $|\lambda (1-v)/(s^\mu(1+\phi s^{-\rho})^\gamma)|<1$ (blue region,
					with $x=\Re(s)$, $y=\Im(s)$) for $(\lambda,v,\mu,\phi,\rho,\gamma)
					=(1,0.5,0.5,1,0.25,1)$.}
			\end{figure}

			\subsection*{Acknowledgements}
			
				\v{Z}ivorad Tomovski has been supported by the Berlin
				Einstein Foundation through a Research Fellowship during his visit to the
				Weierstrass Institute for Applied Analysis and Stochastics in Berlin for
				three months during 2013.

				Federico Polito has been supported by project AMALFI (Universit\`{a} di Torino/Compagnia di San Paolo).

	\end{document}